\documentclass[a4paper,10pt]{amsart}
\usepackage{yfonts}

\title{Vector bundles on Fano varieties of genus ten}

\author{ Micha\l\ Kapustka\qquad}
\address{Institute of Mathematics, Jagiellonian University of Krak\'ow,
ul.\L ojasiewicza 6, 30-348 Krak\'ow, Poland}
\email{Michal.Kapustka@im.uj.edu.pl}
\thanks{This project was completed while the first author visited
University of Oslo supported by an EEA Scholarship and Training fund
in Poland}
\author{ Kristian Ranestad}
\address{Department of Mathematics, University of Oslo,
PB 1053 Blindern, 0316 Oslo, Norway}
\email{ranestad@math.uio.no}
\keywords{Moduli of vector bundles, Fano varieties,  K3 surfaces}
\subjclass[1991]{Primary: 14D20; Secondary: 14J45, 14J28}

 \theoremstyle{plain} \newtheorem{thm}{Theorem}[section]
 \newtheorem{lemm}[thm]{Lemma}
\newtheorem{prop}[thm]{Proposition}

\newtheorem{cor}[thm]{Corollary}
\theoremstyle{definition}

\newtheorem{rem}[thm]{Remark}

 \numberwithin{equation}{section}
\numberwithin{equation}{section}
\theoremstyle{definition}

\begin{document}
\begin{abstract}In this note we describe a unique linear embedding of a prime Fano 4-fold $F$
    of genus 10 into the Grassmannian $G(3,6)$. We use this to construct some moduli spaces of bundles on sections of $F$.
    In particular the moduli space of bundles with Mukai vector $(3,L,3)$ on a generic polarized K3 surface $(S,L)$ is constructed as a
    double cover of $\mathbb{P}^2$ branched over a smooth sextic.

\end{abstract}
\maketitle
\section{Introduction}
Mukai showed that a polarized Fano $4$-fold $(F,L)$ of genus $10$ and index $2$ has a unique linear embedding as a hyperplane
section of the homogeneous variety $G_2\subset \mathbb{P}^{13}$,  the closed
orbit of the adjoint representation of the
simple Lie Group $\mathbb{G}_2$ \cite[ch 22]{FH}.  This is part of his famous linear section
theorem \cite [thm 2]{Mu1} and \cite {Mu2}, on smooth varieties whose general linear curve section
is a canonical curve of genus at most $10$.
Kuznetsov \cite[sec 6.4, App. B]{Kuz} showed that every hyperplane section of $G_{2}$ admits a pair of possibly isomorphic vector bundles
of rank $3$ with $6$ independent sections.
In this paper we show
\begin{thm}\label{embedding} Any smooth hyperplane section of $G_2$
admits a unique linear embedding as a linear section in the Grassmannian $G(3,6)$.
\end{thm}
Together with Mukai's theorem this means that any smooth Fano fourfold of genus 10 and index 2 has the above property.
We also note that the variety $G_2$ does not admit a linear embedding
into $G(3,6)$ (cf. Corollary \ref{G not on g36}).
As a consequence, we get explicit models for moduli spaces of vector bundles on linear sections of $F$.
In particular, we construct moduli spaces of vector bundles on generic
K3 surfaces and Fano 3-folds of genus 10.
More precisely, let $(S,L)$ be a general K3 surface and let $(X,L)$
be a general prime Fano 3-fold of genus 10. By Mukai's linear
section theorem \cite{Muk1} the surface $S$ and the 3-fold $X$ are complete codimension
$3$ (resp. $2$)  linear section of $G_2$. For a linear section $Y$ of
$G_{2}$, we let
$\Pi_{Y}\subset (\mathbb{P}^{13})^*$ be the linear space orthogonal to the
linear span of $Y$.  A point in $\Pi_{Y}$ is therefore represented by a
hyperplane section $F$ of $G_2$ that contains $Y$.
 Now a generic point $[F]$ on $\Pi_{Y}$ corresponds to a linear embedding
of $Y$ into $G(3,6)$ coming from the embedding of $F$.
When $Y$ is a general complete linear section of dimension at least
$1$, we show that the pullbacks of the two universal bundles on
$G(3,6)$ by this embedding are stable vector bundles.

Let $M_S(3,L,3)$ denote the moduli space of stable sheaves $E$ on $S$
with Mukai vector
$$({\rm rk}E,c_1(E),\frac {1}{2}c_{1}(E)^2-c_{2}(E)+{\rm rk}E)=({\rm rk}E,c_1(E),\chi(E)-{\rm rk}E)=(3,L,3).$$
 We show
\begin{thm}\label{MS}
 The space $M_S(3,L,3)$ is a double cover of the plane $\Pi_{S}$ branched over the
intersection of $\Pi_{S}$ with the sextic dual variety
$\hat{G_2}$.  In particular, $M_S(3,L,3)$ is a K3 surface with a genus
$2$ polarization.
\end{thm}

Thus $S$ and $M_S(3,L,3)$ is an explicit geometric example of a Mukai dual pair of $K3$ surfaces (cf. \cite{MukK3}).
Furthermore, let
$M_{X}(3,L,\sigma,2)$ be the moduli space of stable rank $3$ vector
bundles $E$ on $X$ with Chern classes $c_{1}(E)=L$, $c_{2}(E)=\sigma$
where $\sigma$ is the class of a curve of degree $9$ and genus 2 on $X$
and ${\rm deg}(c_{3}(E))=2$.
\begin{thm}\label{MX}
     The moduli space $M_{X}(3,L,\sigma,2)$ has an irreducible
     component $M_X$ that is a double cover of the line $\Pi_{X}$ branched over the
intersection of $\Pi_{X}$ with the sextic dual variety
$\hat{G_2}$.  In particular, $M_{X}$ is a genus
$2$ curve.
\end{thm}

In the proof of Theorem \ref{embedding} we study the Hilbert scheme of conic sections
and cubic surface scrolls in $F$ and $G_2$.  In particular we observe
that if the line through two points on $G_2$ is not contained in $G_2$,
then there is a unique conic section in $G_2$ through the two points.
Combining this with J. Sawons and D. Markushevichs results on Lagrangian fibrations on
the Hilbert scheme of points on a K3 surface, see \cite{Sawon} and \cite{DM}, we prove
\begin{cor}\label{h2}  The Hilbert scheme $H(X)$ of conic sections on a generic Fano
    $3$-fold $X$ of  genus 10 is isomorphic to the Jacobian of the genus $2$ curve $M_{X}$.
    \end{cor}
    A similar example is worked out in \cite{IR}.

In section 2 we formulate the main theorem and show an explicit embedding of $F$ into $G(3,6)$.
In section 3 we prove the uniqueness of the embedding by analyzing conic sections on $F$
which appear to be the zero loci of the generic sections of the considered bundles.
The key is to prove the following result on the Hilbert scheme of conic sections and of cubic surface scrolls on $F$.
\begin{prop}\label{conics and scrolls in F}  The Hilbert scheme of conic sections on $F$ is isomorphic to the graph of the Cremona transformation on $\mathbb{P}^5$
 defined by the linear system of quadrics in the ideal of a Veronese surface.

  The Hilbert scheme of cubic
 surface scrolls on $F$ is isomorphic to the union of two disjoint projective planes.
\end{prop}
We extend this analysis to get the following result suggested to us by Frederic Han.
\begin{prop}\label{lines in F}  The Hilbert scheme of lines on $F$ is isomorphic to a smooth divisor of bidegree $(1,1)$ in $\mathbb{P}^2\times\mathbb{P}^2$.
\end{prop}
In section 3 we prove stability of the constructed bundles and injectivity of the map to the moduli space.
In section 4 we find a geometric way to describe all linear sections of $G(3,6)$ which are isomorphic to $F$.
More precisely $F$ is obtained as a component of the intersection of $G(3,6)$ with a maximal dimensional linear
space in a quadric of rank 12 containing $G(3,6)$. The second component being a $\mathbb{P}^1\times \mathbb{P}^1\times \mathbb{P}^1$.
An analogous construction is valid for the variety $G_2$ in $G(2,7)$, but there the second component is $\mathbb P^2\times \mathbb{P}^2$.
In the last section we deduce the results concerning moduli spaces of bundles on sections of $F$.

\section{The embedding}\label{sec embedding}
This section is devoted to proving the existence part of Theorem \ref{embedding} and analyzing the constructed example.
First we need a result on the orbits of the adjoint action of the simple Lie group $\mathbb{G}_2$.
\begin{lemm}\label{pencil of orbits} The projectivized adjoint representation of
$\mathbb{G}_2$ admits a pencil of invariant sextic hypersurfaces 
$\mathcal{D}=\{D_{\mu}\}_{\mu\in \mathbb{P}^1}$.
 Let $B$ be the base locus of this pencil.  
\begin{enumerate}
 \item There is a distinguished element $D_l\in \mathcal{D}$
 which is the discriminant variety of the adjoint variety
 \item For each sextic $D_{\mu}\in \mathcal{D}$ different from $D_l$ the set $D_{\mu}\setminus B$
  is an orbit of the representation.
\end{enumerate}
\end{lemm}
\begin{proof}  By \cite[Ex.30]{SK}, the Lie
algebra {\gothfamily g}$_2$ is a subalgebra of {\gothfamily
gl}$(7,\mathbb{C})$, given by matrices of the following form:
\begin{displaymath}
A=\left(\begin{array}{ccccccc}g&h&i&0&f&-e&a\\
           j&k&l&-f&0&d&b\\
m&n&-g-k&e&-d&0&c\\
0&-c&b&-g&-j&-m&d\\
c&0&-a&-h&-k&-n&e\\
-b&a&0&-i&-l&g+k&f\\
2d&2e&2f&2a&2b&2c&0
          \end{array}\right).
\end{displaymath}
It is a simple Lie algebra of rank 2 and dimension 14. Let us denote by {\gothfamily t} its Cartan subalgebra spanned by $g,k$. Let
$$Q_x=\operatorname{det}(t.\operatorname{Id}-ad(x))=\sum_{i=1}^{14}\delta_i(x)t^i,$$
be the characteristic polynomial of the adjoint operator (in our case
it is the restriction of the adjoint operator on {\gothfamily
gl}$(7,\mathbb{C})$ to the subalgebra {\gothfamily g}$_2$, which is
given by a $14\times 14$ matrix with linear entries). By \cite[thm
8.25]{Tev} the polynomial $\Delta(x):=\delta_2(x)$ is a polynomial of degree 12
which can be decomposed into $\Delta(x)=\Delta_l(x).\Delta_s(x)$ in such a way that
$\Delta_l$ is the discriminant of the adjoint variety (the
projectivization of the orbit of the highest weight vector) and $\Delta_s$
is the discriminant of the variety
$\overline{\mathbb{P}(\mathcal{O}_s)}$, where $\mathcal{O}_s$ is the
orbit of any short root vector. The polynomials $\Delta_s$ and $\Delta_l$ restricted to {\gothfamily t} are just the products of short and long roots respectively. It follows that they are
irreducible of degrees equal to the numbers of long and short root
vectors respectively, which is $6$ in both cases. They define hypersurfaces $D_{s}$ and $D_{l}$, respectively,
that generate a pencil of invariant sextics,
which proves the first sentence of the lemma together with (1).
Moreover, the ring of invariants of $\mathbb{G}_{2}$ in the coordinate ring
 of the adjoint representation is $\mathbb{C}[Q,\Delta_{l}]$, where $Q=48(ad+be+cf)+16(g^2+k^2+(g+k)^2+jh+im+nl))$ is the quadric defined by the Killing form.
It follows that we have a $\mathbb{C}^2$ of closed orbits.
Using in addition the fact that all orbits of the representation which are not contained in the discriminant
are closed, this proves that the codimension 1 orbits fill the complement of the discriminant.
\end{proof}
\begin{rem}\label{dual to g2}
The sextic $D_l$ is identified with the dual variety of $G_2$ via the
Killing form and by \cite{Hol} it's singular locus is of codimension two.
\end{rem}

\begin{proof}[Proof of existence in Theorem \ref{embedding}]
Consider
first the whole variety $G_2$.
    It is a linear section of the Grassmannian $G(5,V)$ in it's
Pl\"ucker embedding parameterizing $5$-spaces
    in a $7$-dimensional vector space $V$ isotropic with respect to a
non-degenerate four-form $\omega$.  The span of $G_2$ is the projectivization of
 the adjoint representation of the simple Lie group $\mathbb{G}_2$.
From lemma \ref{pencil of orbits} this projectivized representation has a one-dimensional family of codimension 1 orbits.

In the above coordinates $(a,\dots,n)$ for the Lie-algebra $\mathbb P^{13}=\mathbb P($\text{\gothfamily g}$_2)$,
the variety $G_2$ is defined by the $4\times 4$ Pfaffians of the matrix
$$\left(\begin{array}{ccccccc}
0&-f&e&g&h&i&a\\
f&0&-d&j&k&l&b\\
-e&d&0&m&n&-g-k&c\\
-g&-j&-m&0&c&-b&d\\
-h&-k&-n&-c&0&a&e\\
-i&-l&g+k&b&-a&0&f\\
-a&-b&-c&-d&-e&-f&0
  \end{array}\right).
$$
Consider now a pencil of hyperplane sections given
by the hyperplanes $g=-\lambda (g+k)$.
Observe that this pencil corresponds by the Killing form to the projectivization of the Cartan subalgebra {\gothfamily t} from the proof of Lemma \ref{pencil of orbits}. The latter meets $D_l$ in the set of long roots while $D_s$ in the set of short roots. As these are disjoint, the corresponding line do not meet $B=D_l\cap D_s$. It follows now from Lemma \ref{pencil of orbits} that this pencil has a representative in each orbit of codimension 1 which is not the one contained in the discriminant $D_l$.

For the above pencil of hyperplane sections we explicitly construct a family of embeddings as linear sections
of the Grassmannian $G(3,6)$. To do this it is enough to observe that for each $\lambda \neq 0,-1$ the ideal generating the section of $G_2$ is also generated by the Grassmann quadrics corresponding to the following data.

$$ d, \left(\begin{array}{ccc}
 c& b&-\lambda e\\
-t&l&a\\
m&j&\lambda c
\end{array}\right),
\left(\begin{array}{ccc}
 -(1+\lambda)f&-e&b\\
-\lambda (1+\lambda)t&\lambda n&d\\
-(1+\lambda )i&-h&-f
\end{array}\right)
,(1+\lambda)a,$$
where $t=g+k=\frac{-g}{\lambda}$.
\end{proof}

\begin{rem} We can also find a different family of embeddings of sections of $G_2$ given by the hyperplanes $j=e+\lambda a$ for $\lambda \neq 0$. The embeddings are then described by the data
$$ f, \left(\begin{array}{ccc}
 \lambda b-d& e+\lambda a&b+\lambda l\\
e+\lambda a&g&a+\lambda i\\
b&a&l
\end{array}\right),$$
$$
\left(\begin{array}{ccc}
 h+\lambda b&-k&d-\lambda b\\
-k-\lambda f&-e-\lambda a&-c+\lambda ^2 f+\lambda g+\lambda k\\
d&-c&m-\lambda d
\end{array}\right)
,n-\lambda e.$$

For $\lambda=0$ the above map is no more an embedding but a projection from the only singular point $(0,\dots,1,\dots,0)$, where $i=1$, which is a node. The image of this map is then a proper codimension 2 section of $LG(3,6)\subset G(3,6)$ (for more details see \cite{Unpr}).

\end{rem}
The proof of uniqueness is more delicate and will be postponed to section \ref{unique},
 after we have given some results on the geometry of $G_2$ and its smooth hyperplane sections.
We start with the above pencil of examples.

\subsection{The example}\label{example}
 Let us choose a coordinate system $e_1,\dots,e_6$
on a $6$-dimensional vector space $U$ and denote by $e_{ijk}$, with $1\leq i<j< k\leq 6$, the
corresponding Pl\"ucker coordinates on $\mathbb{P}(\wedge^3U)$ . Let
$x_{ijk}$ denote the respective dual linear forms. The considered family
of hyperplane sections of $G_2$ is then described as a family of sections of $G(3,U)$
by the linear spaces $H^\lambda _{12}\subset \mathbb{P}(\bigwedge^3U)$ given by equations:

\begin{eqnarray*}
x_{456}-(1+\lambda)x_{125}=0\\
x_{134}+x_{356}=0\\
x_{126}-\lambda x_{234}=0\\
x_{123}+x_{346}=0\\
x_{124}-\lambda x_{256}=0\\
x_{156}-(1+\lambda)x_{345}=0\\
x_{146}+\lambda(1+\lambda)x_{235}=0\\.
\end{eqnarray*}

Consider $\Pi^\lambda _6\subset  \mathbb{P}(\bigwedge^3U^*)$ the orthogonal $\mathbb{P}^6$ to the space $H^{\lambda} _{12}$ for $\lambda\neq 0,-1$.
 Consider moreover the set of reducible $3$-forms
  $$\Omega=\{\alpha \colon \alpha=\beta \wedge v, \beta \in
(\bigwedge^2(U))^*, v\in U^* \}\subset(\bigwedge^3(U))^*.$$
   We can explicitly compute that for each $\lambda \neq 0,-1$ the intersection $\Pi^{\lambda} _6\cap \Omega$
is a scroll over a Veronese surface in two ways.
   Indeed, $\Omega\setminus G(3,U^*)$ has two canonical projections
onto $5$-dimensional projective spaces
   $\mathbb{P}(U)$ and $\mathbb{P}(U^*)$. Then we check that $\Pi^{\lambda} _6$
do not meet $G(3,U^*)$.
   Now, to see the scroll structure on the first projection, we need
only to solve the system of equations
    $v\wedge l=0$ with parameters $v$,$\lambda$ and indeterminate $l\in\Pi_6$.
It follows by computations in Macaulay 2 \cite{M2}, for details see \cite{comp},
     that for each $\lambda\neq 0,-1$ there is a Veronese surface of $v$'s such that this system
of equations has a nontrivial solution and the set of solution is
then a line.
     Analogously we consider the second projection. The following
follows:
\begin{lemm}\label{lemm G(2,5) cap phi(F) in example}
Let $Y$ be a variety isomorphic to $G(2,5)$ linearly embedded in
$G(3,U)$, then the linear span of $Y$ intersects each linear space  $H^{\lambda} _{12}$
in a plane or a $4$-dimensional projective space.
\end{lemm}
\begin{proof}
 Observe that $Y$ is equal to one of the following
\begin{itemize}
 \item the Grassmannian $G(3,W)$ for some $5$-dimensional vector subspace $W\subset  U$,
 \item the partial Flag variety $F(w,3,U)$ of $3$-dimensional subspaces of $U$ that contains a given vector $w\in U$.
\end{itemize}
In both cases $<Y>^\perp \subset \Omega$ is the closure of a fiber of
one of the two projections of $\Omega\setminus G(3,U^*)$.
 But we checked above that the intersection of $\Pi^{\lambda}_6=(H^{\lambda}_{12})^{\perp}$
is either empty or it is a line. The assertion follows by duality.
\end{proof}
\begin{cor}\label{cor G(2,5) cap phi(F) in example}
 The set of $W$ (resp. $w$) for which the intersection $<G(3,W)>\cap
H^{\lambda}_{12}$ (resp. $F(w,3,U)\cap H^{\lambda}_{12}$ )
 is of dimension $4$ is a Veronese surface in $\mathbb{P}(U^*)$
(resp. $\mathbb{P}(U)$).
\end{cor}
\begin{proof}
The proof follows directly from the proof of the lemma and the
discussion above it.
\end{proof}

\section{Conic sections and cubic surface scrolls on $G_2$ }\label{unique}
In this section we study the Hilbert scheme of conic sections and cubic surface scrolls on $G_2$.
Let $V$ be a $7$-dimensional vector space,  let $\omega\in\wedge^4 V$ be a $4$-form defining $G_2\subset \mathbb{P}(\wedge^2 V)$, and let
$\omega^*$ be the corresponding dual $3$-form.  Then, as in
\cite[fig.1]{ott}  and  \cite[Ex. 30]{SK}, we may choose coordinates
$V^*=<x_{0},...,x_{6}>$ such that  $\omega^*\in \wedge^3V^*$ is defined as
$$\omega^*=x_{0}\wedge x_{1}\wedge x_{2}+x_{3}\wedge x_{4}\wedge
x_{5}+(x_{0}\wedge x_{3}+x_{1}\wedge x_{4}+x_{2}\wedge x_{5})\wedge
x_{6}.$$
The variety $G_2\subset G(2,V)$ is defined by the Pfaffians of the skew
symmetric matrix $(x_{ij})\quad 0\leq i,j\leq 6, x_{ij}+x_{ji}=0$,
restricted to the $2$-vectors that are killed by $\omega^*$.  With
coordinates $x_{ij}=x_{i}\wedge x_{j}$ the variety $G_2$ is then given
 by the Pfaffians of
$$ \left(\begin{matrix}
0&-x_{56}&x_{46}&x_{03}&x_{04}&x_{05}&x_{06}\\
-x_{56}&0&-x_{36}&x_{13}&x_{14}&x_{15}&x_{16}\\
-x_{46}&x_{36}&0&x_{23}&x_{24}&x_{25}&x_{26}\\
-x_{03}&-x_{13}&-x_{23}&0&x_{26}&-x_{16}&x_{36}\\
-x_{04}&-x_{14}&-x_{24}&-x_{26}&0&x_{06}&x_{46}\\
-x_{05}&-x_{15}&-x_{25}&x_{16}&-x_{06}&0&x_{56}\\
-x_{06}&-x_{16}&-x_{26}&-x_{36}&-x_{46}&-x_{56}&0\\

\end{matrix}\right)
$$
with $x_{03}+x_{14}+x_{25}=0$  as in the previous section.

\begin{lemm}\label{planes on G_2}There are no planes on $G_2$.
\end{lemm}
\begin{proof}
We compute directly in any chosen point $p\in G_2$ that the intersection of $G_2$ with the tangent space to $G_2$ in
$p$ is a cone over a twisted cubic curve, hence do not contain a plane. Indeed, let $p=(0,\dots,0,1,0,\dots,0)$ where $x_{05}=1$.
Then the tangent space at $p$ is given by
$x_{13}=x_{14}=x_{16}=x_{23}=x_{24}=x_{26}=x_{36}=x_{46}=0$
and its intersection with $G_2$ is additionally given by the Pfaffians of
$$ \left(\begin{matrix}
0&-x_{56}&0&x_{03}&x_{04}&x_{05}&x_{06}\\
-x_{56}&0&0&0&0&x_{15}&0\\
0&0&0&0&0&x_{25}&0\\
-x_{03}&0&0&0&0&0&0\\
-x_{04}&0&0&0&0&x_{06}&0\\
-x_{05}&-x_{15}&-x_{25}&0&-x_{06}&0&x_{56}\\
-x_{06}&0&0&0&0&-x_{56}&0\\
\end{matrix}\right)
$$
with $x_{03}+x_{25}=0$
i.e. by
the $2\times 2$ minors of
$$ \left(\begin{matrix}
x_{56}&x_{06}&x_{04}\\
x_{15}&x_{56}&x_{06}\\
\end{matrix}\right)
$$
and $x_{03}+x_{25}=x_{03}^2=0$.
Therefore the tangent cone is a cone over a twisted cubic curve. We conclude by homogeneity.
\end{proof}

 \begin{lemm} Let
     $Q\in \mathbb{P}(V^{*})$ be the $\mathbb{G}_{2}$ invariant smooth quadric
     hypersurface, and let $U\subset V$, be a $6$-dimensional
subspace. Then the intersection $G_2\cap G(2,U)$ is a fourfold of
degree $6$ that spans a $\mathbb{P}^7$.  Furthermore
     \begin{itemize}
\item If $[U]\notin Q$, then $G_2\cap G(2,U)$ is a smooth hyperplane
section of  $\mathbb{P}^2\times \mathbb{P}^2$.
\item If $[U]\in Q$, then $G_2\cap G(2,U)$ is a hyperplane section of
$\mathbb{P}(O_{\mathbb{P}^2}(2)\oplus T_{\mathbb{P}^2}(-1)))$.
\end{itemize}
\end{lemm}
\begin{proof}  The group $\mathbb{G}_{2}$ has two orbits on
$\mathbb{P}(V^{*})$, so for the intersection
$G_2\cap G(2,U)$ there are two isomorphism classes.
We compute these directly.  If $U\subset V$ is the $6$-dimensional subspace $x_{6}=0$, then $[U]\notin Q$ and the
intersection $G_2\cap G(2,U)$ is defined by
$\{\alpha\in G(2,U): (x_{0}\wedge x_{1}\wedge x_{2}+x_{3}\wedge
x_{4}\wedge x_{5})(\alpha)=(x_{0}\wedge x_{3}+x_{1}\wedge
x_{4}+x_{2}\wedge x_{5})(\alpha)=0$
so it is defined by the $2\times 2$ minors
of
$$
 \left(\begin{matrix}
x_{03}&x_{04}&x_{05}\\
x_{13}&x_{14}&x_{15}\\
x_{23}&x_{24}&x_{25}\\

\end{matrix} \right)
$$
with $x_{03}+x_{14}+x_{25}=0$, i.e.
a smooth fourfold Fano hyperplane section of  $\mathbb{P}^2\times
\mathbb{P}^2$.
On the other hand if $U\subset V$ is the $6$-dimensional subspace
$x_{0}=0$, then $[U]\in Q$ and the intersection $G_2\cap G(2,U)$ is defined by
$\{\alpha\in G(2,U): (x_{3}\wedge x_{4}\wedge x_{5}+(x_{1}\wedge
x_{4}+x_{2}\wedge x_{5})\wedge x_{6})(\alpha)=
(x_{1}\wedge x_{2}+x_{3}\wedge x_{6})(\alpha)=0,$ so it is defined by
the Pfaffians
of
$$ \left(\begin{matrix}

0&-x_{36}&x_{13}&-x_{25}&x_{15}&x_{16}\\
x_{36}&0&x_{23}&x_{24}&x_{25}&x_{26}\\
-x_{13}&-x_{23}&0&x_{26}&-x_{16}&x_{36}\\
x_{25}&-x_{24}&-x_{26}&0&0&0\\
-x_{15}&-x_{25}&x_{16}&0&0&0\\
-x_{16}&-x_{26}&-x_{36}&0&0&0\\

\end{matrix}\right)
$$
 i.e. by the $2\times 2$ minors
of the symmetric matrix
$$
\left(
\begin{matrix}
x_{24}&x_{25}&x_{26}\\
x_{25}&-x_{15}&-x_{16}\\
x_{26}&-x_{16}&x_{36}\\
\end{matrix}
\right)
$$
 and the three quadratic entries in the product matrix
$$
(x_{13}\quad x_{23}\quad x_{12})\cdot
\left(
\begin{matrix}
x_{24}&x_{25}&x_{26}\\
x_{25}&-x_{15}&-x_{16}\\
x_{26}&-x_{16}&x_{36}\\
\end{matrix}
\right)
$$
together with the linear relation $x_{12}+x_{36}=0$.
But these quadrics define the linear embedding of the
$\mathbb{P}^2$-bundle
$\mathbb{P}(O_{\mathbb{P}^2}(2)\oplus T_{\mathbb{P}^2}(-1))$, so the
lemma follows.

\end{proof}

\begin{lemm}\label{W in V}  Let  $W\subset V$, be a $5$-dimensional
subspace.
     One of the following holds
     \begin{itemize}
\item  the intersection $G_2\cap G(2,W)$ in $G(2,V)$ is proper.  It is a possibly singular conic section.
\item the intersection $G_2\cap G(2,W)$ in $G(2,V)$ is not proper.  It
is a possibly singular cubic surface scroll.
\end{itemize}
There is a $7$-fold subvariety in $G(5,V)$ of $5$-dimensional
subspaces $W\subset V$ such that the intersection $G_2\cap G(2,W)$ is a
cubic scroll.
\end{lemm}

\begin{proof}
   Let $W\subset V$, be a $5$-dimensional
subspace, and let $U$ be a general $6$-dimensional subspace of $V$ that contains $W$.
  We first assume that the intersection $G_2\cap G(2,U)$ is isomorphic
  to a hyperplane section of $\mathbb{P}^2\times \mathbb{P}^2$.
From the above proof, the variety $\mathbb{P}^2\times \mathbb{P}^2\subset G(2,U)$
  may be identified with $2$-dimensional subspaces that has a
  $1$-dimensional intersection with each of $3$-dimensional subspaces $U_{1}$ and $U_{2}$ that together
span $U$.
   The further intersection with $G(2,W)$ therefore depends on whether
$W$ contains $U_{1}$ or $U_{2}$.
    If neither is the case, the intersection $G_2\cap G(2,W)$ is a
linear section of the set of $2$-dimensional
    subspaces that intersect $U_{1}\cap W$ and $U_{2}\cap W$ so it is
a conic section.
    Similarly, if $W\supset U_{1}$, then  $G_2\cap G(2,W)$ is a smooth cubic
surface scroll.

If  $G_2\cap G(2,U)$ is a hyperplane section of
$\mathbb{P}(O_{\mathbb{P}^2}(2)\oplus T_{\mathbb{P}^2}(-1)))$, then
we may use the above equations to identify $G_2\cap G(2,W)$.
In fact if  $U$ is defined by $x_{0}=0$, then $G_2\cap G(2,U)$ spans a
$\mathbb{P}^7$ on which we may choose coordinates
$$x_{12},x_{13},x_{23},x_{15},x_{16},x_{24},x_{25},x_{26},$$ as above, in which
the equations defining the threefold $X_{U}=G_2\cap G(2,U)$ are the $2\times 2$-minors of
the symmetric matrix

$$
\left(
\begin{matrix}
x_{24}&x_{25}&x_{26}\\
x_{25}&-x_{15}&-x_{16}\\
x_{26}&-x_{16}&-x_{12}\\
\end{matrix}
\right)
$$
and the quadrics
$$
(x_{13}\quad x_{23}\quad x_{12})\cdot
\left(
\begin{matrix}
x_{24}&x_{25}&x_{26}\\
x_{25}&-x_{15}&-x_{16}\\
x_{26}&-x_{16}&-x_{12}\\
\end{matrix}
\right).
$$

So $X_{U}$ is a threefold contained in the fourfold cone over a Veronese
surface with vertex a line $L_{U}$. Since the Veronese surface
contains no lines, the only lines in $X_{U}$ are lines that intersect
$L_{U}$ and lie over a unique point on the Veronese surface.  On the
other hand, for each point on the Veronese surface the symmetric
matrix has rank $1$, so the three remaining quadrics restricted to
the corresponding plane in the cone have a common linear factor, so
they define  a line, except in one plane:  In the plane
$P_{U}: x_{15}=x_{16}=x_{24}=x_{25}=x_{26}=0$ the intersection with $X_{U}$
is the double line defined by $x_{12}^2=0$.
Let $W\subset U$ be a $5$-dimensional subspace.  The intersection
$G_2\cap G(2,W)$ is a complete linear section of both $G(2,W)$ and of
$X_{U}$.  Any complete linear surface section of $X_{U}$ is either a
Veronese surface, a cone over a linear section a Veronese surface, or
it contains the vertex line $L_{U}$.  Any complete surface linear
section of $G(2,5)$ that contains a Veronese surface, contains also a
plane.  A surface cone over a linear section of the Veronese surface is
either a quadric or a quartic surface, but neither are contained in
linear sections of $G(2,5)$ that do not also contain a plane, so only
linear sections of $X_{U}$ that contain $L_{U}$ can be linear
sections of $G(2,5)$.   These are easily seen to correspond to
subspaces $W$ defined by $x_{0}=\ell=0$, where $\ell\in
<x_{4},x_{5},x_{6}>$, and for each such $\ell$ the corresponding
intersection $G(2;W)\cap G_2$ is a cubic scroll.
When the intersection $G(2;W)\cap G_2$ is proper in $G(2,V)$ it is
clearly a conic as above.

  Setting up an incidence variety we find,
    $$I=\{(W,U): {\rm dim} G(2,W)\cap G_2>1\}\subset F=\{(W,U):W\subset
U\}$$
    We have seen that the fiber over each $U$ is $2$-dimensional, so
    ${\rm dim}I=6+2=8$.
    The fiber over each $W$ is clearly
$1$-dimensional,
     so there is a $7$-dimensional variety of $5$-dimensional
subspace $W\subset V$ such that $G(2,W)\cap G_2$ is a cubic surface
scroll.

    \end{proof}

    \begin{cor}\label{conic on G} Let $Z$ be a scheme of length $2$ in $G_{2}$ contained in a line $\ell$.
Then one of the following holds:
\begin{enumerate}
 \item The line $\ell$ is contained in $G_2$
 \item There is a unique conic section in $G_2$ that contains
$Z$.
\end{enumerate}
\end{cor}
\begin{proof}
Assume that $\ell$ does not lie in $G_2$.
Then any conic section in $G_{2}$ through $Z$ spans a plane contained in
the linear span of any $G(2,5)$ in $G(2,7)$  that contains the $Z$.
But any $G(2,5)$ intersects $G_2$ in either a conic section or a cubic scroll.
 In either case there is a unique conic section in the intersections that contains $Z$.
 \end{proof}
 \begin{rem} Jarek Buczynski, private communication, showed that this property is common to the closed orbit of the
 adjoint representation of any semisimple Lie group. \end{rem}
  \begin{cor}\label{3 points define G(2,W)}  If three points $p,q,r\in G_2$ do not lie on a conic
    section, there is at most one $5$-dimensional subspace
    $W\subset V$ such that $p,q,r\in G(2,W)$.
    \end{cor}
    \begin{proof} If $W$ and $W'$ are distinct $5$-dimensional
        subspace then $G(2,W)\cap G(2,W')\cap G_2$ is either a line
        or a conic section.
        \end{proof}


     \begin{cor}\label{F cap G(2,5)=G_2 cap G(2,5)}  Let $F$ be a smooth hyperplane section of $G_2$ and
$W\subset V$ a $5$-dimensional subspace.
 Then  $F\cap G(2,W)$ is a conic section if and only if $F\cap G(2,W)=G_2\cap
G(2,W)$ and $G_2\cap G(2,W)$ is a conic section.
$F\cap G(2,W)$ is a cubic scroll if and only if $F\cap G(2,W)=G_2\cap
G(2,W)$ and $G_2\cap G(2,W)$ is a cubic scroll.
    \end{cor}

    \begin{proof} Since $F$ defines a hyperplane section, it either
contains $G_2\cap G(2,W)$ or it intersects it in a hyperplane section.
But $G_2\cap G(2,W)$ is either a conic section or a cubic scroll, so the
corollary follows.
    \end{proof}

         \begin{cor}  Let $L\subset \mathbb{P}^{13}$ be a linear subspace of codimension $2$ or $3$, such that
         $Y= G_2\cap L$ is a smooth $3$-fold (resp. surface) with Picard group
         generated by $H$.  If $F$ and $F'$
         are distinct hyperplane sections of $G_2$ that contain
         $Y$, and $S_{W}$ and $S_{W'}$ are cubic scrolls on
         $F$ and $F'$ respectively, then $S_{W}\cap Y$ and
         $S_{W'}\cap Y$ are either empty or distinct subschemes on $Y$.
    \end{cor}

\begin{prop}\label{embedded F intersected with G(2,5)}
Let $F$ be any smooth hyperplane section of $G_2$. Let $U$ be a
$6$-dimensional vector space and let $u$
and $W$ be a nontrivial vector and a $5$-dimensional subspace of $U$
respectively .
Let $\phi:F\rightarrow G(3,U)$ be a linear embedding. Then the
intersections $\phi(F)\cap G(3,W)$,
$\phi(F)\cap F(u,3,U)$ are linearly isomorphic to one of the
following:
\begin{enumerate}
 \item a conic section
\item a cubic surface scroll in $\mathbb{P}^4$
\end{enumerate}
Moreover every conic section in $\phi(F)$ lies in the intersection of a
unique $F(u,3,U)$ and a unique $G(3,W)$, while every cubic scroll in
$\phi(F)$ lies in a unique $F(u,3,U)$ or a unique $G(3,W)$.
\end{prop}
\begin{proof}
Observe first that $G(3,W)$ and the partial Flag variety $F(u,3,U)$ are $6$-dimensional
Schubert cycles on $G(3,U)$ that intersect the $4$-dimensional
subvariety $\phi(F)$ in dimension at least one.   Furthermore, both
$G(3,W)$ and $F(u,3,U)$ are isomorphic to the Grassmannian $G(2,5)$,
so the considered intersections  are linearly isomorphic to some
linear sections of the Grassmannian $G(2,5)\subset\mathbb{P}^9$.
Therefore, let us consider a  linear subspace $L\subset
\mathbb{P}^9=<G(2,5)>$, and see when $L\cap G(2,5)$ is spanned by a
variety that is linearly isomorphic to a subvariety of $F$ and hence
of $G_2$. We have the following possibilities:
\begin{itemize}
\item $T=L\cap G(2,5)$ is a variety of codimension 3 in $L$. Then
this intersection is proper, hence $\operatorname{deg}(T)=5$ and a generic linear
section of appropriate codimension is a set of five isolated points
spanning a $\mathbb{P}^3$. Such five points in $F\subset G(2,7)$ lie in a unique $G(2,5)\subset G(2,7)$,
so by Lemma \ref{W in V}, the intersection $G(2,5)\cap G_2$ must be a cubic scroll.
  Therefore $F\cap G(2,5)$ has codimension $2$ in its linear span, against the assumption.
\item $T=L\cap G(2,5)$ is a variety of codimension 2 in $L$. Then the
generic intersection with a plane in $L$ is a set of isolated points.
By \cite[prop 2.5]{Muk8} the number of these points cannot be greater
than 3. It follows that $T$ has degree 3. To be contained in $G_2$ it
cannot contain planes nor quadric surfaces, hence it has to be a
smooth or singular cubic surface scroll.
\item $T=L\cap G(2,5)$ is a variety of codimension 1 in $L$. Then as
above $T$ has to be of degree 2 and hence to be contained in $G_2$ it
needs to be a conic section.
\item $L\subset G(2,5)$. Then to be contained in $G_2$ it has to be a
line. But by adjunction a line cannot be the intersection of $G(2,5)$
with $\phi(F)$. Indeed, if $C$ is a smooth curve of intersection of
$G(2,5)$ with $\phi(F)$ then $C$ is linearly isomorphic to the zero
locus of a rank 3 bundle on $F$ with determinant $H$. By adjunction
$K_C=(H+K_F)|_C=-H|_C$ which implies $C$ is a conic section.
\end{itemize}
Hence the first part of the Proposition follows.
To prove the second part we consider each case separately.
Each conic section in $G(3,6)$ is contained in a unique partial Flag variety $F(A,3,B)$ where $A$
and $B$ are fixed vector spaces of dimension $1$ and $5$ respectively.  So the conic section lies in
both $G(3,B)$ and in $F(A,3,U)$.
For a cubic surface scroll $S$ we need the following:
\begin{lemm}\label{scroll in G(3,6)}
 For any cubic surface scroll  $S\subset G(3,U)$ such that $S=<S>\cap G(3,U)$
  there is a subspace $A$ of dimension $1$ or a subspace $B$ of dimension $5$, such that $S\subset F(A,3,U)$ or $S\subset G(3,B)$.
\end{lemm}
\begin{proof}  The union of planes in $\mathbb{P}(U)$ parameterized by $S\subset G(3,U)$ form a variety $X_{S}\subset\mathbb{P}(U)$.
The lines in $S$ defines pencil of planes in $\mathbb{P}(U)$ that each have a common line and fill a $\mathbb{P}^3$ in $X_{S}$.
The common lines form a scroll $Y_{S}\subset X_{S}$.
The degree of $S$ in the Pl\"ucker embedding equals the sum of degrees of $X_{S}$ and $Y_{S}$.  In fact the degree of $S$ is, by Schubert calculus,  the sum of the degree of the
 intersection of $S$ with the cycle of planes that meet a general line, and the general cycle of planes that meet a general $\mathbb{P}^3$ in a line.
 The former is clearly the degree of $X_{S}$, while the latter is the degree of $Y_{S}$, since a general  $\mathbb{P}^3$ meet a plane in $S$ in a line if and only if it meets the common line of the pencil
  that the planes belongs to in a point.
 If $X_{S}$ is not a fourfold, all its $\mathbb{P}^3$s coincide and $<S>\subset G(3,6)$, contrary to the assumption.
 If $X_{S}$ is a fourfold of degree $1$ it spans a hyperplane $\mathbb{P}(W)\subset\mathbb{P}(U)$, and $S\subset G(3,W)$.
 If $X_{S}$ is a fourfold of degree $2$, then $Y_{S}$ has degree $1$, which means that $Y_{S}$ is a plane and the lines in $Y_{S}$ have a common point $u$.  So in this case  $S\subset F(u,3,U)$.
If $X_{S}$ is a fourfold of degree $3$, then $Y_{S}$ is a line, a common line for all the planes in $S$, in which case  $<S>\subset G(3,6)$, contrary to the assumption.
\end{proof}
The intersection
$G(3,B)\cap F(A,3,U)=F(A,3,B)$ so the intersection with $\phi(F)$
cannot contain $S$, so any cubic surface scroll $S\subset\phi(F)$ is contained in precisely one of them.
\end{proof}
\begin{cor}\label{G not on g36} The variety $G_{2}$ does not admit a
    linear embedding into the Grassmannian $G(3,6)$.
    \end{cor}
    \begin{proof} Let $U$ be a $6$-dimensional vector space.  An
embedding $G_{2}\subset G(3,U)$ induces an
embedding of any hyperplane section $F$ of $G_{2}$.  By
Proposition \ref{embedded F intersected with G(2,5)} the
intersection
$F\cap G(3,B)$ is a conic section for a general $5$-dimensional
subspace $B\subset U$. But then  $G_{2}\cap G(3,B)$ must
be a quadric surface, against the fact that there is either a
line or a unique
conic section through two points on $G_{2}$ (cf. Corollary
\ref{conic on G}).
\end{proof}
 \begin{rem}\label{conic on scroll} On any cubic surface scroll $\Sigma$, the Hilbert scheme $H_\Sigma$ of conic sections is a $\mathbb{P}^2$.
 In a cubic cone, the conic sections are all singular while the double lines form themselves a conic section $C_{\Sigma}$ in $H_\Sigma$.
  Any pencil of singular conics on $\Sigma$ with a fixed line form a tangent line to  $C_{\Sigma}$.
 On a smooth cubic scroll the general conic section is smooth, while the singular ones form  a line in $H_\Sigma$.
 \end{rem}

The Proposition \ref{conics and scrolls in F} is an immediate corollary of the following
\begin{prop}\label{hilb scheme of conics} The Hilbert scheme $H_F$ of conic sections
lying on a smooth hyperplane section $F$ of $G_2$ is isomorphic to the graph in $\mathbb{P}_{1}^5\times\mathbb{P}_{2}^5$
of the Cremona transformation $\gamma: \mathbb{P}_{1}^5\cdots >\mathbb{P}_{2}^5$ defined
by the linear system of quadrics that contain a Veronese surface $V_{1}\subset \mathbb{P}_{1}^5.$
The inverse Cremona transformation is defined by the quadrics that contain a Veronese surface $V_{2}\subset \mathbb{P}_{2}^5.$

    The Hilbert scheme $V_{F}$ of cubic surface scrolls in $F$
    has two components each of which is isomorphic to $\mathbb{P}^2$.
    Each component has a natural embedding as the Veronese surface, $V_{i}\subset\mathbb{P}_{i}^5, i=1,2$,
    and the Hilbert scheme of conics $H_{F}$ restricts
    to a $\mathbb{P}^2$-bundle over these two Veronese surfaces.
\end{prop}
\begin{proof} Let $U$ be a $6$-dimensional vector space and let $F\to G(3,U)$ be a linear embedding as in example \ref{example}.
By Proposition \ref{embedded F intersected with G(2,5)} each conic $C$ in $F$ is the intersection
$F(u_{C},3,U)\cap G(3,V_{C})\cap F$ for a unique  one dimensional subspace $<u_{C}>\in U$ and a unique codimension one subspace $V_{C}\subset U$.
Therefore the Hilbert scheme $H_F$ has a natural embedding  in $\mathbb{P}(U)\times \mathbb{P}(U^*)$
   taking a conic $C$ to the unique pair $([u_{C}],[V_{C}])$.
Furthermore, by Lemma \ref{lemm G(2,5) cap phi(F) in example} and Corollary \ref{cor G(2,5) cap phi(F) in example} , for $u\in U$, the intersection $F(u,3,U)\cap F$ either a conic section or a cubic surface scroll
and similarly for $G(3,V)\cap F$ when $V\subset U$ is a codimension one subspace.
  The Hilbert scheme $V_{F}$ of cubic surfaces in $F$ is isomorphic to two Veronese surfaces $V_{1}\subset \mathbb{P}(U)$
   and $\mathbb{P}(U^*)$ via the maps $[u]\mapsto  \mathbb{P}(U)$ and $[V]\mapsto  \mathbb{P}(U^*)$ respectively.

   We can then define a rational map $\gamma:  \mathbb{P}(U)\ni [u] \mapsto [V]\in \mathbb{P}(U^*)$ such that $F(u,3,U)\cap G(3,V)\cap F\in H_F$.
  By the above, the map is birational and defined outside the Veronese surface $V_{1}$, so it is a Cremona transformation.  The fiber in $H_{F}$ over $[u]\in V_{1}$ is
  the Hilbert scheme $H_{\Sigma}$ of conic sections on the cubic scroll $\Sigma= F(u,3,U)\cap F$.  This fiber is therefore isomorphic to $\mathbb{P}^2$,
  which is also the fiber of the projectivized normal bundle of $V_{1}$ in $\mathbb{P}(U)$ .
  Each conic in $H_{\Sigma}$ is mapped to a unique $[V]\in \mathbb{P}(U^*)$ by the second projection, so $H_{F}\subset \mathbb{P}(U)\times \mathbb{P}(U^*)$ is
the graph of the Cremona transformation $\gamma$.  It remains to show that $\gamma$ is the Cremona transformation defined
 by the linear system of quadrics that contain $V_{1}$. By symmetry, the inverse is mapped by the  linear system of quadrics that contain $V_{2}$.
 Observe first that the projections $p_1:H_{F}\to  \mathbb{P}(U)$ and $p_2; H_{F}\to  \mathbb{P}(U^*)$ map the plane of conics $H_{\Sigma}$
  lying on a cubic surface scroll to a point by one projection and to a plane by the other.
 Indeed, let $\Sigma\subset F(u,3,U)$ be a cubic surface scroll for some one dimensional subspace $<u>\subset U$.  Consider the $5$-dimensional quotient $W=U/<u>$ and the cubic scroll
 $\Sigma$ as a subvariety of the Grassmannian $G(2,W)$.
  The lines in $\mathbb{P}(W)$ corresponding to points in $\Sigma$
 all meet a line $L_{\Sigma}$ and are contained in a nodal quadric hypersurface.
Now any $G(2,4)\subset G(2,W)$ intersects the scroll in  conic section if and only if it is defined by a $4$-space that contains $L_{\Sigma}$.
 Therefore this family of Grassmannians $G(2,4)$ is defined by a linear family of codimension one subspaces in $W$,
 which means that the corresponding family of codimension one subspaces $V\subset U$ is a plane in  $\mathbb{P}(U^*)$
 These planes are the image by the second projection of the fibers $H_{\Sigma}\cong\mathbb{P}^2$ in $H_{F}$ over the Veronese surface $V_{1}$.
 Any two of these planes intersect only on the Veronese surface. In particular the union of these planes form a hypersurface.
  The Cremona transformation $\gamma^{-1}$ must contract the planes, while the Veronese surface is the base locus,
    so it must be given by the quadrics defining the Veronese surface.
\end{proof}

We now pass to the proof of uniqueness.
\begin{proof}[Proof of uniqueness in Theorem \ref{embedding} ]
Let $\psi:F\rightarrow G(3,U)$ be any linear embedding. By
Proposition \ref{embedded F intersected with G(2,5)}
this embedding gives rise to two projections from the Hilbert scheme
of conic sections on $F$ onto $\mathbb{P}^5$.
Again by Proposition \ref{embedded F intersected with G(2,5)}
and Proposition \ref{hilb scheme of conics} these projections must
coincide with the projections in the constructed examples. Let $Q$ be the
universal quotient bundle and $E$ the universal bundle on $G(3,U)$.
It follows that there is a linear isomorphism $\pi$ between $H^0(F,
\psi^*(Q))$ and one of the spaces
$H^0(F,\phi^*(Q))$ or $H^0(F,\phi^*(E^*))$ such that the sections $s$
and $\pi(s)$ have the same zero locus.
Therefore $\psi$ is the embedding defined by one of the bundles
$\phi^*(Q)$, $\phi^*(E^*)$.
We conclude that there exists a linear automorphism of $G(3,6)$
mapping $\phi(F)$ to $\psi(F)$.
\end{proof}

We end the section by applying our analysis of conics and cubic scrolls on $F$  to the Hilbert scheme of lines on $F$.
The Proposition \ref{lines in F} was suggested to us by Frederic Han.
\begin{proof}[Proof of Proposition \ref{lines in F}] We keep the notation from the proof of Proposition \ref{hilb scheme of conics}. 
We shall use the word
cubic scroll for both a smooth cubic scroll and a cone over a cubic curve. In the latter case every lines in the scroll is both exceptional and a ruling.
Let $l\subset F\subset G(3,6)$ be a line. There exist spaces $U_l$ and $V_l$ of dimension 2 and 4
such that $l=F(U_l,3,V_l)$. Consider the set $H_l\subset H_{F}$ of conics contained in $F$ and containing $l$.
This set is mapped by both $p_1:H_{F}\to \mathbb{P}(U)$ and $p_2:H_{F}\to \mathbb{P}(U^*)$
 surjectively onto the lines $l_1\subset \mathbb{P}(U)$ and $l_2\subset \mathbb{P}(U^*)$ defined by
 $\{<u>\subset U_l\}$ and $\{V\supset V_l\}$.
Observe that $H_l$ has at most two components each one isomorphic to $\mathbb{P}^1$.
Indeed, each component of $H_l$ is either contracted or maps surjectively onto $l_1$ or $l_2$.
But both $p_1$ and $p_2$ are one to one over the complements of the Veronese surfaces.
It follows that $H_l$ has either two components each contracted by one of the projections
$p_1$ and $p_2$ or $H_l$ has one component which is not contracted in any direction.
Notice, that in the first case, $l_{1}$ and $l_{2}$ are contracted by the Cremona transformation
$\gamma$ and $\gamma^{-1}$ respectively.  Therefore $l_{1}$ and $l_{2}$ must intersect $V_{1}$, respectively $V_{2}$,
in a scheme of length two.
In the latter case the line $l_{1}$ is mapped to the line $l_{2}$
by the Cremona transformation $\gamma:  \mathbb{P}(U)\cdots> \mathbb{P}(U^*)$, so $l_{1}$ and $l_{2}$,
intersect $V_{1}$ and $V_{2}$, respectively, transversely in one point.

Let $l\cup l'\subset F$ be a conic section.  Let $v\in H_{l}$ be its parameter point.
Assume that $H_{l}$ has a component through $v$ that is contracted by the projection $p_{1}$.
Then $l\cup l'$ lies in a cubic scroll $\Sigma_{v_{1}}$, where $v_{1}=p_{1}(v)\in V_{1}$, and $l$ must be contained in a pencil of conic sections on this scroll.
In particular $l$ must be an exceptional line on the scroll.
If $H_{l}$ does not have a component through $v$ that is contracted by the projection $p_{1}$ and $v_{1}=p_{1}(v)\in V_{1}$,
 then $l\cup l'$ belongs to a pencil of conics on the scroll $\Sigma_{v_{1}}$ with $l'$ as a fixed component, i.e. $l'$ is exceptional on the scroll.

When $H_{l}$ has two components, one component is contracted to a point $v_{1}(l)\in V_{1}$.
In this case $v_{1}(l)\in l_{1}\cap V_{1}$, and since  $l_{1}\cap V_{1}$ has length two,
there is a unique point $w_{1}(l)$ residual to $v_{1}(l)$ in $l_{1}\cap V_{1}$.
 Similarly there is a unique point $w_{2}(l)$ residual to $v_{2}(l)$ in $l_{2}\cap V_{2}$.
 When $H_{l}$ has a unique component, we set $w_{i}(l)=l_{i}\cap V_{i}$ for $i=1,2$.
 For each $i$ there is a unique singular conic section $l\cup l(i)\in H_{l}$ which is mapped to $w_{i}(l)$ by  the projection $p_{i}$.
 Clearly $l(i)$ is a line in the cubic scroll $\Sigma_{w_{i}(l)}$, and $l\cup l(i)$ belongs to a pencil of conics with $l(i)$ as a fixed component.

 Summing up, we may define a morphism from the Hilbert scheme $Hilb_l(F)$ of lines to $\mathbb{P}^2\times\mathbb{P}^2:$
  $$\rho:Hilb_l(F)\to V_{1}\times V_{2}$$
  by
  $$l\mapsto (w_{1}(l),w_{2}(l)).$$
  The scroll $\Sigma_{w_{1}(l)}$ contains $l$ as a ruling, and every line $l'$ in this ruling has $w_{1}(l')=w_{1}(l)$.
  Indeed, this $\mathbb{P}^1$ of lines is a fiber of $\rho$ composed with the projection to the first factor.
  Clearly, this fiber is mapped to a line in the second factor:  Indeed, the pencil of
  conic sections in $\Sigma_{w_{1}(l)}$ with fixed component  $l(1)$ form a line in $V_{2}$.  Likewise the fiber of the composition of
   $\rho$ with the second factor is mapped isomorphically to a line in the first factor.

 In this construction clearly $\rho$ is an embedding, and all fibers of the projections to the two
 factors are mapped to isomorphically to lines in the other factor, so $Hilb_l(F)$ is smooth and the proposition follows.
  \end{proof}

\begin{rem} Observe that the intersection of two scrolls from the same family
 may contain a line that is exceptional on one of the scrolls,
 and that the intersection of two scrolls from different families can be a smooth conic.
\end{rem}
\begin{rem} Observe that in the case where $H_l$ has two components if $w_1(l)=v_1(l)$ the scroll $\Sigma_{w_{1}(l)}$ has to be a cone. And analogously for $w_2(l)=v_2(l)$. Moreover, when both $w_1(l)=v_1(l)$ and $w_2(l)=v_2(l)$ the conic corresponding to the intersection of the two components is a double line. Note that $w_1(l)=v_1(l)$ does not imply $w_2(l)=v_2(l)$ and vice versa.
\end{rem}

\section{The bundles}

Recall that $(S,L)$ is a generic $K3$ surface of genus 10 which is a
general linear section of $G_2$. Consider a smooth hyperplane
sections $F$ of $G_2$, containing
$S$. By Theorem \ref{embedding} $F$ correspond to a
pair $(\mathcal{E},\mathcal{E}')$ of bundles of rank 3 each, the
pullbacks of the universal quotient bundle and the dual to the
universal subbundle on $G(3,6)$. We shall prove that all such bundles
are stable and different.
Let us first prove stability.

\begin{prop}\label{stability} The bundles $E$ and $E'$ are $L$-stable.
\end{prop}
\begin{proof} Consider a generic curve $C$ in $|L|$. It is enough to
prove that the bundles are stable on $C$. Observe first that by
construction $E$ is a rank $3$ bundle on $C$ such that $c_1(E)=K_C$,
$h^0(E)=6$. We have two possibilities to destabilize $E$. Either
there is a subbundle of $E$ of rank $1$ and degree $\geq 6$ or a rank $2$
subbundle of degree $\geq 12$. We consider the two cases separately.

Assume first that there is a line bundle $M$ on $X$ of degree $d\geq
6$ which is a subbundle of $E$.
We get the exact sequence:
$$ 0\longrightarrow M\longrightarrow E\longrightarrow N
\longrightarrow 0,$$
with $c_1(N)=K_C\otimes M^{*}$. It follows that $h^0(M)+h^0(N)\geq
h^0(E)=6$. On the other hand by the Riemann-Roch formula we have
$h^0(M)-h^0(K_C\otimes M^{*})= d-9\geq-3$. We claim that $h^0(M)\geq
1$. Indeed, if $M$ had no sections then the map $H^0(E)\rightarrow
H^0(N)$ would be an embedding. Hence the bundle $N$ would be globally
generated by 6 sections. We could then consider a nonvanishing
section of $N$. Its wedge products with the remaining generators
would give 5 linearly independent sections of $c_1(N)$. The latter
contradicts the Riemann-Roch formula. Hence $h^0(M)\geq 1$ and
each of its sections defines a hyperplane in $\mathbb{P}(H^0(E))$
which contains $d$ fibers of the natural image of $\mathbb{P}(E)$ in
$\mathbb{P}(H^0(E))$. It follows that $C$ intersects a $G(3,W)$ (for
some hyperplane $W\subset V$) in $d\geq 6$ points. By Proposition
\ref{embedded F intersected with G(2,5)} this implies that $C$ is
either a conic section or is contained in a cubic scroll. This contradicts
genericity.

Assume now that there is a rank 2 bundle $N$ on $X$ of degree $d\geq
12$ which is a subbundle of $E$.
We get the exact sequence:
$$ 0\longrightarrow N\longrightarrow E\longrightarrow M
\longrightarrow 0.$$
Then $M=K_C\otimes c_1(N)^{*}$ has at most one section by assumption
on $C$. Hence, $h^0(N)\geq 5$. But we have proved above that $E$ does
not admit any line subbundle with 2 sections (as by assumption on $C$
it would have degree $>6$), hence neither does $N$. It follows that
the projectivization of the kernel of the map $\bigwedge^2
H^0(N)\rightarrow H^0(\bigwedge^2 N)=H^0(c_1(N))$ does not meet the
Grassmannian $G(2,H^0(N))$. Finally $h^0(c_1(N))\geq 7$ which
contradicts the Riemann Roch for $M$.
\end{proof}
Let us now prove that all these bundles are different.
\begin{prop}\label{injectivity} Consider two smooth hyperplane sections $F_1$ and $F_2$ of $G_2$ containing $S$. Denote by
$\mathcal{E}_1,\mathcal{E}'_1, \mathcal{E}_2$ , and
$\mathcal{E}'_2$ their corresponding bundles. Then no two of these
bundles are isomorphic.
\end{prop}

Before we pass to the proof of the proposition let us prove the
following lemmas.

\begin{lemm}\label{cubic scroll in G(2,7) is in G(2,5)} For every cubic
scroll $S$ contained in $G(2,V)$ such that $S=<S>\cap G(2,V)$,
there is a $5$-dimensional subspace $W\subset V$ such that
$S\subset G(2,W)$.
\end{lemm}
\begin{proof}  The union of lines in $\mathbb{P}(V)$ parameterized by $S\subset G(2,V)$ form a variety $X_{S}\subset\mathbb{P}(V)$.
The lines in $S$ define pencils of lines in $\mathbb{P}(V)$ that each have a common point and fill a $\mathbb{P}^2$ in $X_{S}$.
The common points form a curve $Y_{S}\subset X_{S}$.
The degree of $S$ in the Pl\"ucker embedding equals the sum of degrees of $X_{S}$ and $Y_{S}$.  In fact the degree of $S$ is, by Schubert calculus,  the sum of the degree of the
 intersection of $S$ with the cycle of lines that meet a general $\mathbb{P}^3$, and the general cycle of lines that is contained in a general $\mathbb{P}^5$.
 The former is clearly the degree of $X_{S}$, while the latter is the degree of $Y_{S}$, since a general  $\mathbb{P}^5$ contains a line in $S$ if and only if it meets the common point of the pencil
  that the line belongs to in a point.
 If $X_{S}$ is not a threefold, all its $\mathbb{P}^2$s coincide and $<S>\subset G(2,V)$, contrary to the assumption.
 If $X_{S}$ is a threefold of degree $1$ it is a $\mathbb{P}^3\subset\mathbb{P}(U)$, and $<S>\cap G(2,V)$ is a quadric hypersurface, contrary to the assumption.
 If $X_{S}$ is a threefold of degree $2$, then $X_{S}$ is a quadric hypersurface in a hyperplane $\mathbb{P}(W)\subset\mathbb{P}(U)$ and $S\subset G(2,W)$.

If $X_{S}$ is a threefold of degree $3$, then $Y_{S}$ is a point, i.e. all the lines parameterized by $S$ have a common point, so $<S>\subset G(2,V)$, contrary to the assumption.
\end{proof}

\begin{lemm} \label{degenerate sections of bundles vanish on scrolls}
If $s$ is a section of one of the bundles
$\mathcal{E}_1,\mathcal{E}'_1, \mathcal{E}_2$ , and $\mathcal{E}'_2$
then its zero locus on $S$ is the intersection of $S$ with $G(2,W)$
for some $W$ of dimension $5$. In particular it is either empty or is a
$0$-dimensional scheme of length $2$ or $3$ spanning a line or a plane
respectively. Moreover in the case of length $3$ the subspace $W$ is uniquely
determined.
\end{lemm}
\begin{proof} Observe that all considered bundles admit exactly six
global sections. Hence, these global sections are restrictions of
global sections of appropriate universal bundles on $G(3,6)$. It
follows that the zero loci of these sections are intersection of the
image of $S$ by the embedding considered in Theorem $\ref{embedding}$
with some $G(2,5)\subset G(3,6)$ (there are two types of such as in
Proposition \ref{embedded F intersected with G(2,5)}). By Proposition
\ref{embedded F intersected with G(2,5)} these are linear sections of
conic sections or cubic scrolls by space of codimension 2. As both a conic section
and a cubic scroll are contained in $G(2,W)$ for some $W$ of
dimension 5, and $S$ do not contain neither a conic section nor a twisted
cubic curve we get the assertion together with the classification of
examples. For the last statement we use Corollary \ref{3 points define G(2,W)}
\end{proof}

\begin{proof}[Proof of Proposition \ref{injectivity}]
Observe first that $S$ does not contain any twisted cubic curve. It
follows that $F_1$ and $F_2$ do not contain any common cubic scroll.
We prove that $\mathcal{E}_1$ is different from the other
bundles, the rest being analogous. Let $D$ be a cubic scroll in $F_1$
from the chosen family corresponding to $\mathcal{E}_{1}$. Then $D\cap
S$ is the zero locus of a section of $\mathcal{E}_1$. By Lemma
\ref{degenerate sections of bundles vanish on scrolls} it is a
$0$-dimensional scheme of length 3 which is contained in $G(2,W)$ for a
unique $W\subset V$ of dimension 5. We claim that it is not the zero
locus of any section of the remaining bundles. This is a consequence
of the fact that $D$ is contained neither in $F_2$ nor in the other
family of scrolls on $F_1$ and Corollary \ref{F cap G(2,5)=G_2 cap
G(2,5)} and proves that the bundles are not isomorphic.
\end{proof}

\section{Geometry and general construction}\label{sec Geometry and general construction}
In this section we shall give a geometric construction of $F\subset G(3,6)$.
Let us start by analyzing possible constructions coming from geometry
of the variety $G_2$.
We shall describe the space spanned by $G_2$ in terms of quadrics
containing $G(2,V)$, for a $7$-dimensional vector space $V$.
By \cite[prop. 1.3.]{Muk8} the linear system of quadrics vanishing on
$G(2,V)$
is naturally isomorphic to $P^*(\bigwedge^4V)$. The orbits on this
space under the natural $Sl(7)$ action are described in \cite{ott}.
\begin{lemm}\label{7 quadrics}
There is a $7$-dimensional vector space of  quadrics in the ideal of the
Grassmannian $G(2,V)$ vanishing on the linear span of $G_2$, and all of them have
rank 12. Moreover every variety which is a linear section of
$G(2,V)$ with this property is isomorphic to $G_2$.
\end{lemm}
\begin{proof}  Let $\omega\in\wedge^4 V$ be a $4$-form defining $G_2=\{\alpha\in G(2,W):
\alpha \wedge \omega=0\}\subset \mathbb{P}(\wedge^2 V)$, and let
$\omega^*$ be the corresponding dual $3$-form.
The forms $\beta\in\bigwedge^4V^{*}$
define quadric forms $q_{\beta}$ that vanish on $G(2,V)$ in the
following sense $q_{\beta}(\alpha)=\beta(\alpha\wedge \alpha)$. We easily
check that the quadrics vanishing on the span of $G_2$ correspond to
the wedge product of $\omega^{*}$ with linear forms.  Dually, this
wedge product corresponds to reductions of $\omega$ be the linear
coordinates
of $V$.

To prove the second assertion, let us consider a $7$-dimensional
linear space $L$ of quadrics of rank $12$.
We shall prove that there exists a form $\omega$ such that it's
reductions modulo the coordinates generate $L$.
Observe that the seven
quadrics of rank $12$ correspond to $4$-forms
$x_1\wedge\omega_1,\dots,x_7\wedge\omega_7$ for some
$x_1,\dots,x_7\in V$,
$\omega_1,\dots,\omega_7\in \bigwedge^3V$. From the assumption we
have that for any $\lambda_1,\dots,\lambda_7\in\mathbb{C}$ there exist
$x_8\in V$, $\omega_8\in\bigwedge^3V$ such that:
$$\lambda_1 x_1\wedge\omega_1 +\dots+\lambda_7 x_7\wedge\omega_7=x_8
\wedge \omega_8.$$
Observe moreover that no two forms corresponding to quadrics in $L$
have the same linear entry in the above description. Indeed, the line
joining two forms $v\wedge\tilde{\omega}_1$ wedge
$v\wedge\tilde{\omega}_2$ cuts the set corresponding to quadrics of
smaller rank. Hence $x_1,\dots,x_7$ form a basis of $V$. We can then
easily prove by induction that the above conditions imply that there
exits a form $\omega$ such that $x_i\wedge\omega_i =x_i\wedge \omega$
for all $i\in\{1,\dots,7\}$.
\end{proof}
It follows that $G_2$ is contained in a linear subspace in each one of
the above quadrics.
 Note that a rank $12$-quadric in $\mathbb{P}^{20}$ has two families of maximal
dimensional linear spaces of dimension $14$.

We may now formulate the following statement.
\begin{prop}\label{description of G2 as space in quadric of rank 12}
Let $Q$ be a generic quadric of rank $12$ containing the Grassmannian
$G(2,V)$.
Then there are two isomorphic families of $14$-dimensional projective
spaces contained in $Q$, say $\mathcal{F}$ and $\mathcal{G}$, such
that the following holds:
\begin{itemize} \item If $R\in \mathcal{F}$, then $R\cap G(2,W)$ is
linearly isomorphic to $G_2\cup D$, where $D$ is the intersection of
the singular locus of $Q$ with $G(3,W)$ and is isomorphic to
$\mathbb{P}^2\times \mathbb{P}^2$.
\item If $R\in \mathcal{G}$, then $R\cap G(2,W)$ is linearly
isomorphic to a Fano fivefold of degree $24$ containing $D$.
\end{itemize}
\end{prop}
\begin{proof}
By \cite[prop. 1.3.]{Muk1} the quadrics containing $G(2,V)$ correspond to 3-forms on $V$.
By \cite[fig. 1]{ott} the generic quadric of rank 12 (i.e. $Q$) corresponds in a suitable coordinate
system $v_1,\dots,v_6$ to the $3$ form
$v_1\wedge v_2\wedge v_3+v_4\wedge v_5\wedge v_6$. We can then
recover the equation of $Q$ and check directly
 that the singular locus meets $G(2,V)$ in a fourfold $Z$ linearly
isomorphic to $\mathbb{P}^2\times\mathbb{P}^2$.
 Now, let $R$ be a generic (from one of the two families)
  $14$-dimensional projective space contained in $Q$ and let $T$ be a
 generic $15$-dimensional space containing it that intersects $G(3,V)$ properly. Then $T$ intersects $Q$
in $R$ and another linear space $R'$, and
 the intersection of $T$ with $G(3,V)$ has two components $X_1\subset
R$ and $X_2\subset R'$, both of dimension $5$. The union $X_1\cup X_2$
 has degree $42$. Observe that $Z$ is
contained in exactly one of the two varieties $X_1$ and $X_2$ say
  $X_2$ and in both $R$ and $R'$. Now $X_1\cap Z=X_1\cap X_2\cap Z$
is contained in a linear section of $Z$, hence
  $X_1$ does not span the whole space $R$, i.e. it is contained in a
$13$-dimensional linear subspace. Moreover $R\cap R'=(X_1\cap X_2)\cup Z$
  is a hyperplane section of $X_2$ and $X_1\cap X_2$ is a hyperplane
section of $X_1$. It follows that $\deg(X_1)=18$ and $\deg(X_2)=24$.
  Using the standard example we have $X_1$ is smooth. Finally we get
by adjunction that $K_{X_1}=-3H$. It is therefore a Fano
$5$-fold of index $3$, degree $18$ and Picard number $1$. By the
theorem of Mukai (see \cite[thm. 5.2.3]{fano}) it is a section of $G_2$ hence is isomorphic to $G_2$.
\end{proof}

\begin{rem} An alternative way of proving the above proposition is to perform a dimension count.
Indeed, we can set up an incidence relation containing pairs each consisting of a
quadric and a $P^{13}$ contained in it meeting $G(2,7)$ in a variety
linearly isomorphic to $G_2$. The family of quadrics of rank 12 is
$25$ dimensional (see \cite[fig. 1]{ott}) each contains a $15$-dimensional
family of $14$-dimensional subspaces, hence the incidence has
dimension $40$. From the other side let us start with $G_2$. In $G(2,W)$
there is a $34$-dimensional family of varieties linearly isomorphic to
$G_2$ (parameterized by 3-forms on $W$). Now each of them is contained in a $6$-dimensional family of
quadrics which is in the incidence. It follows that each variety
obtained in the construction is linearly isomorphic to $G_2$.
\end{rem}
Let us now pass to the description of $F$ in $G(3,U)$. We shall use the following.
\begin{lemm} All nontrivial quadrics containing $<F>\cup G(3,U)$ are of rank 12.
\end{lemm}

\begin{proof}It is easy to see that there is a surjective map from $U\otimes U^*$ to the space of quadrics generating $G(3,U)$. The map is given by: $U\otimes U^*\ni u\otimes v \mapsto (\bigwedge^3 U\ni \alpha \mapsto (\alpha\wedge u)(v)\wedge \alpha \in \bigwedge^6 U)\in S^2(\bigwedge^6(U))$.  We check directly in the introduced basis for the constructed family of examples that for each $\lambda$ the space of matrices corresponding to the space of quadrics containing $F^{\lambda}=H^{\lambda}\cap G(3,U)$is as follows:
\begin{displaymath}
\left(\begin{array}{cccccc}(\lambda+1)B&0&-\lambda(\lambda+1)F&(\lambda+1)E&-\lambda D&-C\\
           0&(\lambda+1)B&0&A&0&-G\\
D&0&0&0&0&-E\\
-(\lambda+1)G&(\lambda+1)C&0&0&\lambda A&-\lambda F\\
-F&0&0&-C&\lambda B&0\\
-A&(\lambda+1)E&\lambda(\lambda+1)G&\lambda D&0&\lambda B
          \end{array}\right).
\end{displaymath}
Then we check possible eigenforms to see that they all correspond to rank 12 quadrics.
\end{proof}
It follows that we have a $7$-dimensional space of quadrics of rank $12$ containing
$<F>\cup G(3,U)$. We also observe that the generic rank $12$
 quadric corresponds to a matrix with diagonal eigenform and three
(i.e. three pairs of) distinct eigenvalues. It follows
 that the singular set of such a quadric meets $G(3,U)$ in a variety
linearly isomorphic to the Segre embedding of
 $\mathbb{P}^1\times\mathbb{P}^1\times\mathbb{P}^1$. Now we can copy
the proof of Proposition
 \ref{description of G2 as space in quadric of rank 12} in this
context using our constructed example instead of
 the standard description of $G_2$. This proves the following.

\begin{prop}\label{description of F as space in quadric of rank 12}
Let $Q$ be a generic quadric of rank 12 containing the Grassmannian
$G(3,U)$. Then there are two isomorphic families of $13$-dimensional
 projective spaces contained in $Q$, say $\mathcal{F}$ and
$\mathcal{G}$ such that the following holds:
\begin{itemize} \item If $R\in \mathcal{F}$, then $R\cap G(3,U)$ is
linearly isomorphic to $F\cup D$, where $D$ is the intersection of
the singular locus of $Q$ with $G(3,W)$ and is isomorphic to
$\mathbb{P}^1\times \mathbb{P}^1\times \mathbb{P}^1$.
\item If $R\in \mathcal{G}$, then $R\cap G(2,W)$ is linearly
isomorphic to a Fano threefold of degree 24 containing $D$.
\end{itemize}
\end{prop}

\begin{rem}
Performing a dimension count in this context we get a two parameter subgroup of $\operatorname{Sl}(6)$ acting on $F$.
Indeed, let us first compute the dimension of the family of
quadrics of rank 12 containing $G(3,W)$. The generic such quadric
has a singular $\mathbb{P}^8$ meeting $G(3,6)$ in a
$\mathbb{P}^1\times \mathbb{P}^1\times \mathbb{P}^1$ corresponding
to three skew lines in $W$. Triples of such lines are
parameterized by a $24 $-dimensional space. We easily check (by
Macaulay 2  \cite{M2}) that the projection of $G(3,6)$ from such a
$\mathbb{P}^8$ is a complete intersection of two quadrics. It
follows that we have a $25 $-dimensional family of quadrics. Each
considered quadric has a $15$-dimensional (two components) family of
$13$-dimensional projective spaces. Now each smooth Fano fourfold of genus 10 and index 2 contained in $G(3,W)$ is contained in a $6$-dimensional space of quadrics.
 As there is a one-parameter family of such Fano varieties the group $\operatorname{Sl}(6)$ acting on $G(3,6)$ makes a $33$-dimensional family  of varieties linearly isomorphic to $F$.
\end{rem}
\begin{cor} The projections from the span of $\mathbb{P}^1\times \mathbb{P}^1\times \mathbb{P}^1\subset G(3,W)$ and from the span of $\mathbb{P}^2\times \mathbb{P}^2\subset G(2,V)$ define birational maps from $F$ to quadrics in $\mathbb{P}^5$.
\end{cor}
\begin{proof} We check directly, with 
Macaulay 2  \cite{M2}, that the projection from $\mathbb{P}^1\times \mathbb{P}^1\times \mathbb{P}^1$ maps
$G(3,W)$ birationally to a complete intersection of a pencil of quadrics in $\mathbb{P}^{11}$,
and likewise that the projection of $G(2,V)$ from $\mathbb{P}^2\times \mathbb{P}^2$ defines a rational fibration of $G(2,V)$ over a complete intersection of a pencil of quadrics in $\mathbb{P}^{11}$.
From the constructions above it follows that the corresponding $F$ and $G_2$ map through these maps to quadrics in $\mathbb{P}^5$. The assertion follows.
\end{proof}
\begin{rem} The two pencils of quadrics appearing in the above proof are distinct pencils of quadrics in $\mathbb{P}^{11}$.
The image of the projection of $G(3,6)$ is defined by a pencil of quadrics of rank 12 with three degenerate elements of rank 8,
whereas the pencil of quadrics defining the image of the projection of $G(2,V)$ has two degenerate elements of rank 6.
\end{rem}
\section{The moduli space}
In this section we construct the announced moduli spaces. We start with a generic $K3$ surface $(S,L)$
of genus 10. Then $S$ is a proper section of $G_2$ by a $10$-dimensional linear space.
 The variety $G_2$ spans a $13$-dimensional
space $R$, hence there is a plane $\Pi_{S}\subset R^*$ of hyperplanes
of $R$ containing $S$. By Theorem \ref{embedding} a generic
hyperplane $H\in \Pi_{S}$, outside a smooth sextic curve, corresponds
each to a unique embedding of $S$ in $G(3,U)$. Such an embedding
gives rise to two stable and
distinct bundles of rank 3 with a $6$-dimensional vector space of
sections and determinant $L$ by Proposition \ref{stability} and  Proposition \ref{injectivity} respectively.
Since $L$ is nondivisible
every semi-stable bundle in $M_S(3,L,3)$ is stable, therefore, by \cite{MukK3},
the moduli space $M_S(3,L,3)$ is a smooth $K3$ surface.
By our Proposition \ref{description of F as space in quadric of rank 12} and Proposition \ref{injectivity}, the family of
bundles we obtain is a $2$-dimensional algebraic family of pairwise non-isomorphic bundles. It follows that this family
of bundles is an open subset of the K3 surface $M_S(3,L,3)$.
We constructed a rational map from this surface to $\mathbb{P}^2$ which is a two-to-one morphism over
the complement of a smooth sextic.
It follows that the K3 surface is a double cover of $\Pi_{S} \simeq\mathbb{P}^2$ branched over $\Pi_{S} \cap \hat{G_2}$,
where $\hat{G_2}$ denotes the dual hypersurface to $G_{2}$. Indeed,
by \cite{Unpr} the generic point on the sextic dual to $G_2$
corresponds to a nodal hyperplane section which admits a unique
projection from its node to $LG(3,6)$. It follows that for a chosen $S$
the sextic $\Pi_{S}\cap \hat{G_2}$ parameterizes bundles which are pullbacks
by these projections of the universal quotient bundle on $LG(3,6)$.
Notice that on $LG(3,6)$ the universal quotient bundle is isomorphic
to the dual of the universal sub-bundle.

For a general Fano 3-fold $(X,L)$ of genus 10, we do not know a
general structure theorem for the moduli space of stable bundles, so
we do not aim for a complete description in our case.  Still, in
analogy with the surface case we get an irreducible component of the
moduli space.
We start with $X$ as a proper section of $G_2$ by a $11$-dimensional linear space.
As above, the pencil of hyperplane sections $F$ of $G_{2}$ that
contains $X$ gives rise to a pencil of pairs of rank $3$ vector
bundles in $M_{X}(3,L,\sigma,2)$ that are pairwise distinct.  For the six singular hyperplane
sections in the pencil, the two bundles coincide, so we get a
complete family of vector bundles.  To complete the proof of \ref{MX}
we need only show that this family is dense in its closure.  But for each
bundle $E$ in the family the
natural map $\phi_{E}:\wedge^3H^0(X,E)\to H^0(X,L)$ is surjective, and for the general $E$ in the family,  $\phi_{E}$
 even induces an embedding of $X$ into $G(3,6)$.  These two properties
clearly define open conditions on the moduli space
$M_{X}(3,L,\sigma,2)$.   So it suffices to show that if
$E$ is a bundle in $M_{X}(3,L,\sigma,2)$, such that $\phi_{E}$
 is surjective and defines an embedding of $X$ into $G(3,6)$, then  $E$ is the
restriction of a bundle on a Fano $4$-fold $F$ that contains $X$.
By abuse of notation we
denote again by $X$ the image of the embedding  into $G(3,6)$.
Then $E$ is the restriction to $X$ of the universal quotient bundle on $G(3,6)$.
Let $S$ be a generic hyperplane section of $X$.  By the above,  the surface $S$ is
the codimension $2$ linear section of a unique Fano $4$-fold $F$ of genus $10$ and index $2$ in
$G(3,6)$.  Furthermore, this $4$-fold $F$ is a complete linear
section of $G(3,6)$ with a $\mathbb{P}^{12}$.
We claim that $X$ is a hyperplane section of $F$.
Assume it is
not.  Then the linear span $H_{13}=<F\cup X>$ is a $\mathbb{P}^{13}$.
Consider a general quadric hypersurface $Q$ in $\mathbb{P}^{19}=<G(3,6)>$ that
contains $H_{13}\cup G(3,6)$.
We observed  above that all quadrics containing
$<F>\cup G(3,6)$ are of rank $12$.  Since $H_{13}=<F\cup X>$ contains  $<F>$ we conclude that $Q$ has rank $12$.
Let $H_{14}$ be a general $\mathbb{P}^{14}$ that contains $H_{13}$.
Then $H_{14}\cap Q$ splits into the union of $H_{13}$
and another projective space $H'_{13}$ of dimension $13$. By
genericity $H'_{13}$ does not contain $X$: The singular locus of $Q$ is the base locus of the system of $H'_{13}s$ in $Q$ obtained as $H_{14}$ varies.
This singular locus $Sing(Q)$ is a linear space of dimension $7$, and hence does not contain $X$.
Now, every component of $H_{14}\cap G(3,6)$  has dimension at least $4$, so every component of
 $H_{13}\cap G(3,6)$ which is not contained in $Sing(Q)$ must also have dimension at least $4$.
 On the other hand $F$ is a complete hyperplane section of $H_{13}\cap G(3,6)$, so the latter have dimension at most $5$.
This leaves only two possibilities:
\begin{itemize}
\item $H_{13}\cap G(3,6)$ is a $5$-fold $F_5$.  Since $F$ is smooth, it can only have isolated singularities. Then a generic hyperplane section of
$F_5$ containing $X$ is also a smooth Fano $4$-fold  $F'$ of degree $18$ and index $2$.  But $S$ is contained in $X$ and therefore also in $F'$, contrary to the unicity of $F$.
\item $H_{13}\cap G(3,6)$ is a $4$-fold with at least two components, each of degree $18$ and none containing $Sing(Q)\cap G(3,6)$,
which is a threefold in a $\mathbb{P}^7$. It follows that the generic element of the system $H'_{13}$ intersects $G(3,6)$ in a $4$-fold of degree $\leq 42-36=6$, so
the generic element of the system $H_{13}$ in $Q$ intersects $G(3,6)$ in a $4$-fold of degree $\geq 36$.
Consider the intersection $H_{13}\cap H'_{13}\cap G(3,6)$.
It is a hyperplane section of both $H_{13}\cap G(3,6)$ and $H'_{13}\cap G(3,6)$. Then $H'_{13}\cap G(3,6)$ must also have degree at least 36
which gives a contradiction.
\end{itemize}

This concludes the proof of Theorem \ref{MX}.

We end this section with the proof of Corollary \ref{h2}, a corollary of our study of the Hilbert
scheme of conic sections in a different direction.  Consider  first
$S^{[2]}=Hilb^2S$, the Hilbert scheme of pairs of points on a general
K3 surface section of $G_{2}$.  We may assume that $S$ contains no
lines, so by Lemma \ref{conic on G} there is a unique conic section
through any pair of, possibly infinitesimally close, points. We may
also assume that $S$ contains no conic sections, so each conic section lies
in a unique Fano 3-fold section $X$ of $G_{2}$ that contains $S$.
Thus we obtain a morphism $S^{[2]}\to \mathbb{P}^2$.  The fiber $F_{X}$ over $X$ is precisely
the Hilbert schemes of conic sections contained in
$X$.
 Sawon proved,  \cite[Theorem 2]{Sawon}, that $S^{[2]}$
has a Lagrangian fibration to $\mathbb{P}^2$. Markushevich identified the target $\mathbb{P}^2$ with the linear system of the ample
generator of the Picard group on $M_{S}(3,L,3)$.   He  proved, \cite[Theorem 4.3]{DM},  that the general fiber is the
 Jacobian of the corresponding curve in the linear system.  In our case the linear system is $|M_{X}|$, so we conclude that
 the Hilbert scheme of conics sections on $X$ coincides with the Jacobian of the genus $2$ curve $M_{X}$.

\section*{Acknowledgements}
The work was done during the first author's stay at the University of
Oslo between March 2009 and  March 2010. We would like to thank G. Kapustka, J. Weyman, A. Kuznetsov, F. Han, S. Mukai, J. Buczy\'nski, D. Anderson, and L. Manivel for discussions and remarks.

\end{document}